\documentclass[12pt,reqno]{amsart}
\usepackage{amssymb}
\usepackage{txfonts}
\usepackage{amsmath}

 \newtheorem{thm}{Theorem}[section]
 \newtheorem{cor}[thm]{Corollary}
 \newtheorem{lem}[thm]{Lemma}
 
 \newtheorem{exam}[thm]{Example}
 \theoremstyle{definition}
 \newtheorem{defn}[thm]{Definition}
 \theoremstyle{remark}
 
 \numberwithin{equation}{section}

\begin{document}

\title
{A mean value inequality for the generalized self-expander type submanifolds and its application}

\author{Liang Cheng}

\dedicatory{}
\date{}


\thanks{Liang Cheng's Research partially supported by the Natural Science
Foundation of China 11201164,11171126,11301400 and a scholarship from the China Scholarship Council 201308420219.}

 \maketitle

\begin{abstract}
In this paper we get a version of  mean value inequality for generalized self-expander type submanifolds  in Euclidean space.
As the application, we prove that if  mean curvature flow $M(t)$ on the self-expander in Euclidean space subconverges to an $n$-rectifiable varifold $T$ in weak sense for $t$ goes to the singular time,
then $T$ must be the cone.
\end{abstract}

\section{introduction}

Let $x_0:M^n\to \mathbb{R}^{n+m}$ be a complete smooth immersed submanifold in Euclidean space. Consider the mean curvature
flow
\begin{equation}\label{mcf}
\frac{\partial x}{\partial t}=\vec{\mathbf{H}},
\end{equation}
with the initial data $x_0$, where $\vec{\mathbf{H}}=-H\nu$ is the mean curvature vector and $\nu$ is the outer unit
normal vector. The self-similar solutions, including self-shrinkers, translators and self-expanders ,are one of the important subject in the study of mean curvature flow. For other works for studying the  self-similar solutions of the mean curvature flow, one may see \cite{CL}, \cite{CM1},\cite{MV}, and \cite{MV}.
 Recall 
\begin{defn}
 The immersed submanifold $x:M^n\to \mathbb{R}^{n+m}$ is called self-expanders of mean curvature flow if it satisfies
 \begin{equation}\label{self-expander}
 \vec{\mathbf{H}}=\mu(x-p_0)^{\perp},
 \end{equation}
  for some fixed vector $p_0\in \mathbb{R}^{n+m}$ and nonnegative constant $\mu\geq 0$.
\end{defn}
\noindent The mean curvature flow $x(\cdot,t)$ on the self-expander (\ref{self-expander}) satisfying
 \begin{equation}\label{self-expander_1}
 x(p,t)=\sqrt{1+2\mu (t-\frac{1}{2\mu})}(x(\phi_t^*(p))-p_0),
 \end{equation}
with $\phi_t$ being the tangent diffeomorphisms on $M^n$ with
$$x(p,\frac{1}{2\mu})=x(\phi_{\frac{1}{2\mu}}^*(p)),$$
and $Dx(\frac{\partial}{\partial t}\phi_t^*(p))=-\frac{\mu}{1+2\mu(t-\frac{1}{2\mu})}(x(\phi_t^*(p))-p_0)^T$.

Notice that the mean curvature flow always blows up at finite time.
For noncompact hypersurfaces, solution to the mean curvature flow may exist for all times.
For example, Ecker and Huisken \cite{EH1} showed that the mean curvature flow on locally Lipschitz continuous entire graph in
Euclidean space exists for all time.
The self-expanders appear as the singularity model of the mean curvature flow which exists for long time.
For the entire graphs have the bounded gradient and
\begin{equation}\label{growth_condition}
\langle x_0,\nu \rangle^2 \leq c(1+|x_0|^2)^{1-\delta}
\end{equation}
at time $t = 0$, where  $c<\infty$ and $\delta>0$, Ecker and Huisken \cite{EH}
 proved the normalized mean curvature flow
 \begin{equation}\label{normalized_mcf}
\frac{\partial \widetilde{x}}{\partial s}=\vec{\widetilde{ \mathbf{H}}}-\widetilde{x},
\end{equation}
with initial data $x_0$, obtained under the rescaling
\begin{equation}\label{scaling}
\widetilde{x}(\cdot,s)=\frac{1}{\sqrt{2t+1}}x(\cdot,t),\ \  s=\frac{1}{2}\log(2t+1),
\end{equation}
 converges as $s\to \infty$ to a self-expander.

In this paper, we study the following generalized version of self-expanders for the mean curvature flow:
\begin{defn}
 The immersed submanifold $x:M^n\to \mathbb{R}^{n+m}$ is called generalized self-expander type submanifold if it satisfies
 \begin{equation}\label{self-expander}
 \vec{\mathbf{H}}\geq\mu(x-p_0)^{\perp},
 \end{equation}
\end{defn}

We first prove the following mean value inequality.

\begin{thm}\label{thm_21}
Let $x:M^n\to \mathbb{R}^{n+m}$ be the submanifold in the Euclidean space satisfying $\vec{H}\cdot (x-p_0) \geq \mu |(x-p_0)^{\perp}|^2$
for some fixed vector $p_0\in \mathbb{R}^{n+m}$ and nonnegative constant $\mu\geq 0$. Set $M=x(M^n)$.
Then
\begin{align}\label{eq_1}
\frac{d}{dR}(\frac{1}{R^n}\int_{M\cap B_R(p_0)} f)
\geq &\frac{d}{dR}\int_{M\cap  B_R(p_0)}\frac{|(x-p_0)^{\perp}|^2}{|x-x_0|^2}f+\frac{1}{2R^{n+1}}\int_{M\cap B_R(p_0)}(R^2-|x-p_0|^2)\Delta f \nonumber\\
&+\frac{\mu}{R^{n+1}}\int_{M\cap B_R(p_0)} |(x-p_0)^{\perp}|^2f,
\end{align}
for any smooth nonnegative function $f$ on $M$.
Moreover, the equality of (\ref{eq_1}) holds if and only if $x$ satisfies $\vec{H}\cdot (x-p_0) = \mu |(x-p_0)^{\perp}|^2$.
\end{thm}

Next we give an application to the mean value inequality (\ref{eq_1}).
Recall that if $x_0$ is the graphical cone, then the solution to the mean curvature flow (\ref{mcf}) must be the self-expander.
The argument is this (see \cite{I3}): the rescaled
flow $x_j(\cdot,t)=\sqrt{\lambda_j}x(\cdot,\lambda_j^{-1} t)$ solves (\ref{mcf}) with the same initial condition $x_0$ since $x_0$ is a cone, so it must be equal to $x(\cdot,t)$ by the uniqueness of the solution of graphically initial data (if the uniqueness fails, then one may not have the mean curvature flow coming out of cone is the self-expander,see \cite{ACI} ).
On the contrary, a natural question is that on what conditions the corresponding mean curvature flow (\ref{self-expander_1}) for the self-expanders converges
to the cone as $t\to 0$?

\begin{thm}\label{m_1}
 Let $x:M^n\to \mathbb{R}^{n+m}$ be the self-expander (\ref{self-expander}) in the Euclidean space. Let $x(\cdot,t)$ be the corresponding solution (\ref{self-expander_1}) to mean curvature flow for the self-expander $x$. Set $M(t)=x(M^n,t)$. If $M(t)$
  subconverges  to an $n$-rectifiable varifold $T$ in weak sense for some sequence $t
  _j\to 0$, then $T$ must be a cone.
\end{thm}

The structure of this paper is as follows. In section 2, we give the proof of Theorem \ref{thm_21}. Then we give two corollaries of  Theorem \ref{thm_21} as the direct applications to Theorem \ref{thm_21}.
In section 3, we give the proof of Theorem \ref{m_1}.

\section{monotonicity formula and mean value inequality}

First we give the proof of Theorem \ref{thm_21}.

\begin{proof}[Proof of Theorem \ref{thm_21}]
With losing of generality, we may assume $p_0=0$ and $B_R(p_0)=B_R(0)$.
We have
\begin{align}\label{eq_11}
\int_{M\cap B_R(p_0)} div(\frac{x}{R}f)&=-\int_{M\cap B_R(p_0)}\vec{H}\cdot \frac{x}{R}f +\int_{M\cap \partial B_R(p_0)} \frac{x}{R}f\cdot \nu\nonumber\\
&\leq -\mu\int_{M\cap B_R(p_0)} \frac{|x^{\perp}|^2}{R}f +\int_{M\cap \partial B_R(p_0)} f |\nabla |x||,
\end{align}
where we use $\nu=\frac{x^T}{|x^T|}$ and $ \frac{x}{R}\cdot \nu=|\nabla |x||$. Then by coarea formula
we have
\begin{align*}
\frac{d}{dR}\int_{M\cap B_R(p_0)} f&=\int_{M\cap \partial B_R}\frac{f}{|\nabla |x||}\\
\geq&\int_{M\cap \partial B_R}f\frac{1-|\nabla |x||^2}{|\nabla |x||}+\int_{M\cap B_R(p_0)} div(\frac{x}{R}f)+\mu\int_{M\cap B_R(p_0)} \frac{|x^{\perp}|^2}{R}f\\
=&\frac{d}{dR}\int_{M\cap  B_R}f\frac{|x^{\perp}|^2}{|x|^2} +\frac{n}{R}\int_{M\cap B_R(p_0)}f+ \int_{M\cap B_R(p_0)}\frac{x^T}{R}\cdot \nabla f\\
&+\mu\int_{M\cap B_R(p_0)} \frac{|x^{\perp}|^2}{R}f,
\end{align*}
where we use $1-|\nabla|x||^2=\frac{|x^{\perp}|^2}{|x|^2}$. Since $x^T=-\frac{1}{2}\nabla (R^2-|x|^2)$,
we conclude that
\begin{align*}
\frac{d}{dR}\int_{M\cap B_R(p_0)} f
\geq&\frac{d}{dR}\int_{M\cap  B_R}f\frac{|x^{\perp}|^2}{|x|^2} +\frac{n}{R}\int_{M\cap B_R(p_0)}f+\frac{1}{2R}\int_{M\cap B_R(p_0)}(R^2-|x|^2)\Delta f \\
&+\mu\int_{M\cap B_R(p_0)} \frac{|x^{\perp}|^2}{R}f.
\end{align*}
It follows that
\begin{align*}
\frac{d}{dR}(\frac{1}{R^n}\int_{M\cap B_R(p_0)} f)
\geq&\frac{d}{dR}\int_{M\cap  B_R}f\frac{|x^{\perp}|^2}{|x|^2}+\frac{1}{2R^{n+1}}\int_{M\cap B_R(p_0)}(R^2-|x|^2)\Delta f \\
&+\frac{\mu}{R^{n+1}}\int_{M\cap B_R(p_0)} |x^{\perp}|^2f.
\end{align*}
In view of (\ref{eq_11}), the equality of (\ref{eq_1}) holds if and only if $\int_{M\cap B_R(p_0)}\vec{H}\cdot \frac{x}{R}f=\mu\int_{M\cap B_R(p_0)} \frac{|x^{\perp}|^2}{R}f$ for any smooth function $f$ on $M$. It follows that (\ref{eq_1}) holds if and only if $x$ satisfying $\vec{H}\cdot (x-p_0) = \mu |(x-p_0)^{\perp}|^2$.
\end{proof}

As the corollary to Theorem \ref{thm_21}, we have the following mean value inequality.

\begin{cor}\label{cor_1}
Let $x:M^n\to \mathbb{R}^{n+m}$ be the submanifold in the Euclidean space satisfying and $\vec{H}\cdot (x-p_0) \geq \mu |(x-p_0)^{\perp}|^2$
for some fixed vector $p_0\in \mathbb{R}^{n+m}$ and nonnegative constant $\mu\geq 0$. Assuming that $B_{R_0}(p_0)\cap M\neq \emptyset$ with $M=x(M^n)$. If $f$ is a nonnegative function with
$\Delta f\geq -(\lambda+2\mu|(x-p_0)^{\perp}|^2) R_0^{-2}f$, then the function
\begin{equation}\label{eq_21}
  g(R)=e^{\frac{\lambda R}{2 R_0}}\frac{\int_{M\cap B_R(p_0)} f}{R^n}
\end{equation}
is monotone non-decreasing for any $0\leq R\leq R_0$. In particular, if $p_0\in M$, then
\begin{equation}\label{eq_22}
 f(p_0)\leq e^{\frac{\lambda }{2 }}\frac{\int_{M\cap B_{R_0}(p_0)
 } f}{\text{Vol}(B_1\subset \mathbb{R}^n)R_0^n}.
\end{equation}
\end{cor}
\begin{proof}
  It follows from (\ref{eq_1}) that
  $$g'(R)\geq -\frac{\lambda}{2}R_0^{-2}R^{1-n}\int_{M\cap B_R(p_0)}f=-\frac{\lambda}{2}R_0^{-2}Rg(R).$$
  We have $g'(R)\geq -\frac{\lambda}{2}R_0^{-2}R\geq -\frac{\lambda}{2R_0 }g(R)$. Then (\ref{eq_21}) and (\ref{eq_22}) follow immediately.
\end{proof}

For the another corollary to Theorem \ref{thm_21}, by taking $f=1$ in (\ref{eq_1}), we have

\begin{cor}\label{cor_2}
Let $x:M^n\to \mathbb{R}^{n+m}$ be the submanifold in the Euclidean space satisfying $\vec{H}\cdot (x-p_0) \geq \mu |(x-p_0)^{\perp}|^2$
for some fixed vector $p_0\in \mathbb{R}^{n+m}$ and nonnegative constant $\mu\geq 0$. Then
\begin{align}\label{monotonicity_2}
\frac{d}{dR}(\frac{\mathcal{H}^n(M\cap B_R(p_0))}{R^n})
\geq\frac{d}{dR}\int_{M\cap  B_R(p_0)}\frac{|
(x-p_0)^{\perp}|^2}{|x-p_0|^2}
+\frac{\mu}{R^{n+1}}\int_{M\cap B_R(p_0)} |(x-p_0)^{\perp}|^2,
\end{align}
 with the equality  holds if and only if $x$ satisfies $\vec{H}\cdot (x-p_0) = \mu |(x-p_0)^{\perp}|^2$.
\end{cor}

\section{proof of Theorem \ref{m_1}}
Before the proof of Theorem \ref{m_1}, we need the following lemma.
\begin{lem}\label{thm_converges}
Let $x:M^n\to \mathbb{R}^{n+m}$ be the submanifold in the Euclidean space satisfying $\vec{H}\cdot (x-p_0) \geq 0$
for some fixed vector $p_0\in \mathbb{R}^{n+m}$.
 If $\lambda_j(M-p_0)$  subconverges an $n$-rectifiable $T$ in weak sense for some sequence $\lambda_j\to 0$, then $T$ is an $n$-rectifiable cone.
\end{lem}
\begin{proof}
With losing of generality, we may assume $p_0=0$ and $B_R(p_0)=B_R(0)$.
We first prove that the varifold $T$ satisfies the following monotonicity formula
\begin{align}\label{3_1}
t^{-n}\mu_T(B_t(0))-s^{-n}\mu_T(B_s(0))\geq \int_{(B_t(0)\backslash B_s(0))\times G(n,n+m)}r^{-n}|\nabla_{\omega^N}r|^2dT(x,\omega),
\end{align}
for any $0<s<t$, where $r=|x|$ and $\omega^N$ denote the orthogonal $m$-plane to $\omega$.
The proof of (\ref{3_1}) is similar to the case of stationary varifolds (see Proposition 3.7 in).

Let $\phi$ be a nonnegative cutoff function with $\phi'(s)\leq 0$ which is one on $[0,\frac{1}{2}]$ and supported on $[0,1]$. Denote $\eta(r)=\phi(\frac{r}{s})$ so that $r\eta'(r)=-s\frac{d}{ds}(\phi(\frac{r}{s}))$.
Since $\lambda_jM$  subconverges to $T$ in weak sense and $\vec{H}\cdot x \geq 0$ on $M$,
we have
$$\int div_{\omega}(\eta(r)x) dT(x,\omega)=\lim\limits_{j\to\infty}\int div_{\omega}(\eta(r)x) dT_j(x,\omega)=-\lim\limits_{j\to\infty}\int \vec{H}_j\cdot (\eta(r)x) d\mu_j\leq 0,$$
where $\vec{H}_T$ is the the mean curvature vector with respect to the varifold $T$
and $\vec{H}_j$ is the the mean curvature vector with respect to the varifold $\lambda_j M$.
Then
\begin{align*}
0&\geq\int div_{\omega}(\eta(r)x) dT(x,\omega)\\
&=\int(n\eta(r)+r\eta'(r))dT(x,\omega)-\int r\eta'(r)|\nabla_{\omega^N}r|^2dT(x,\omega).
\end{align*}
It follows that
\begin{align*}
\int n\phi(\frac{r}{s})-s\frac{d}{ds}(\phi(\frac{r}{s}))dT(x,\omega)\leq -s\int \frac{d}{ds}(\phi(\frac{r}{s}))|\nabla_{\omega^N}r|^2dT(x,\omega).
\end{align*}
Hence
\begin{align*}
\frac{d}{ds}(s^{-n}\int \phi(\frac{r}{s})dT(x,\omega))\geq s^{-n}\frac{d}{ds}\int \phi(\frac{r}{s})|\nabla_{\omega^N}r|^2dT(x,\omega).
\end{align*}
Let $\phi$ increase to the characteristic function of $[0,1]$, we conclude (\ref{3_1}) holds. From Corollary \ref{cor_2} and taking $\mu=0$, we get $t^{-n}\mu_T(B_t(0))\equiv C$. Then $T$ is the cone in view of (\ref{3_1}).
\end{proof}

Now we give the proof of Theorem \ref{m_1}.
\begin{proof}[Proof of Theorem \ref{m_1}]
Since $\phi_t$ are the tangent diffeomorphisms on $M^n$, we
$$M(t)= x(M^n,t)=\sqrt{1+2\mu (t-\frac{1}{2\mu})}(x(\phi_t^*((M^n))-p_0)=\sqrt{1+2\mu (t-\frac{1}{2\mu})}(x(M^n)-p_0).$$
Then Theorem \ref{m_1} holds by Lemma \ref{thm_converges}.
\end{proof}

\thanks{\textbf{Acknowledgement}: Part of the work is done while the author was visiting the  Math Department of Rutgers University. The author would like to thank Professor Xiaochun Rong for his hospitality and Professor Natasa Sesum for her useful talking about mean curvature flow. }


\begin{thebibliography}{99}

\bibitem{ACI}

S. B. Angenent, D. L. Chopp, and T. Ilmanen.\emph{ A computed example of nonuniqueness of mean
curvature flow in R3}. Comm. Partial Differential Equations, 20 (1995), no. 11-12, 1937¨C1958
\bibitem{B}
J. Bode;  \emph{Mean Curvature Flow of Cylindrical Graphs}. Ph.D. thesis (2007), Freie
Universit\"{a}t Berlin, Universit\"{a}tsbibliothek

\bibitem{CL}
Cao H D, Li H.\emph{ A gap theorem for self-shrinkers of the mean curvature flow in arbitrary codimension}. Calculus of Variations and Partial Differential Equations, 2013, 46(3-4): 879-889.

\bibitem{CS}
Cheng L, Sesum N. \emph{Asymptotic behavior of Type III mean curvature flow on noncompact hypersurfaces}. arXiv preprint arXiv:1403.0235, 2014.

\bibitem{CZ}
Cheng, L; Zhu, A \emph{On the weighted forward reduced volume of Ricci flow.} Proc. Amer. Math. Soc.  141  (2013),  no. 8, 2859-2868.


\bibitem{CM1}
Colding T H, Minicozzi II W P.\emph{ Generic mean curvature flow I; generic singularities}. Annals of Mathematics, 2012, 175(2): 755-833.

\bibitem{EH}
K.Echer, G.Huisken, \emph{mean curvature evolution of entire graphs}. Ann. Math. 130 (1989). 453-471

\bibitem{EH1}
K.Ecker, G.Huisken, \emph{Interior estimates for hypersurfaces moving
by mean curvature}. Invent. Math. 105 (1991), 547-569



\bibitem{H1}
G.Huisken, \emph{Asymptotic behavior for singularities of the mean curvature flow}. J. Differential Geom. 31
(1990), no. 1, 285-299.

\bibitem{HS1}
G. Huisken and C. Sinestrari, \emph{Mean curvature flow singularities for mean
convex surface}, Calc. Var. PDE, 8(1999), 1-14.

\bibitem{HS2}
G. Huisken and C. Sinestrari, \emph{Convexity estimates for mean curvature
flow and singularities of mean convex surfaces}, Acta Math., 183(1999).
47-70.






\bibitem{I1}
Tom Ilmanen, \emph{Singularites of mean curvature flow of surfaces}, preliminary version, available
under http://www.math.ethz.ch/~ilmanen/papers/sing.ps.

\bibitem{I2}
Tom Ilmanen, \emph{Elliptic regularization and partial regularity for motion by mean curvature}, Mem.
Amer. Math. Soc. 108 (1994), no. 520, x90.

\bibitem{I3}

Tom Ilmanen, \emph{Lectures on mean curvature flow and related equations}, lecture notes, ICTP,
Trieste, 1995, http://www.math.ethz.ch/ilmanen/papers/pub.htm

\bibitem{MV}
Ma L, Vicente M.\emph{Bernstein theorem for translating solitons of hypersurfaces}. arXiv preprint arXiv:1405.3042, 2014.

\bibitem{MSS}
Martin F, Savas-Halilaj A, Smoczyk K. \emph{On the topology of translating solitons of the mean curvature flow}. arXiv preprint arXiv:1404.6703, 2014.













\bibitem{W1} B. White, \emph{Stratification of minimal surfaces, mean curvature flows, and harmonic maps}. J. Reine
Angew. Math. 488 (1997), 1-35.
\end{thebibliography}
\end{document}